\documentclass[a4paper,11pt]{amsart}
\title[Projective representations]{Projective representations of fundamental groups of quasiprojective varieties: a realization and a lifting result.}
\author{Ga\"el Cousin}

\usepackage{multirow,yfonts,amsthm}
\usepackage{dsfont}
\usepackage[applemac]{inputenc}
\usepackage[english]{babel}
\usepackage{amssymb,amsmath, amsfonts}
\usepackage{stmaryrd}
\usepackage{xypic}
\usepackage{graphicx}
\usepackage{chngpage}
\usepackage[toc,page]{appendix} 
\usepackage{subfigure}
\usepackage{color}

\definecolor{gris}{gray}{0.45}

\newtheorem{thm}{Theorem}[section]
\newtheorem{lem}[thm]{Lemma}
\newtheorem{prop}[thm]{Proposition}

\theoremstyle{definition}
\newtheorem{defi}[thm]{Definition}

\theoremstyle{remark}
\newtheorem{rem}[thm]{Remark}

\newcommand {\N}{\mathbb{N}}

\newcommand{\C}{\mathbb{C}}

\begin{document}
\maketitle
\begin{abstract}We discuss two results about projective representations of fundamental groups of quasiprojective varieties. The first is a realization result which, under a nonresonance assumption, allows to realize such representations as monodromy representations of flat projective logarithmic connections. The second is a lifting result: any representation as above, after restriction to a Zariski open set and finite pull-back, can be lifted to a linear representation.
\end{abstract}
\section{Introduction}
In this note, we study projective representations $\rho : \pi_1(X\setminus H)\rightarrow \mathrm{PGL}_m(\C)$, for $X$ a projective complex variety and $H$ an algebraic hypersurface in $X$.

If $X$ is smooth and $H$ is normal crossing, under some nonresonance assumption, we show that $\rho$ can be realized as the monodromy representation of a \textsl{flat logarithmic projective connection}; we refer to this as the \textsl{realization result}.
 This allows to extend to $X$ the analytic $\mathbb{P}^{m-1}$-bundle over $X\setminus H$ which underlies the suspension of $\rho$. Thanks to this and an algebraization result of Serre \cite{SerreChev58}, we can derive a second result (\textsl{lifting result}):
with no smoothness and normal crossing assumptions for $X$ and $H$, any $\rho : \pi_1(X\setminus H)\rightarrow \mathrm{PGL}_m(\C)$ is the projectivization of a linear representation, up to adding components to $H$ and pulling back by a generically finite morphism $Y\rightarrow X$. 
Contrary to the first, this second result is not new; it is a well known fact in \'etale cohomology that any class in ${H}^2(X\setminus H,\mathbb{Z}/m\mathbb{Z})$ can be made trivial after the two operations mentioned above.  Yet, it seems of interest to show how it can be derived quickly from the realization result.

We plan to use the lifting result in a future paper about algebraic isomonodromic deformations.

The proof of the realization result is an adaptation of the work of Deligne \cite{MR0417174} on the Riemann-Hilbert problem. For explicitness of basic ideas in this field, we will refer to \cite{MR2077648}.
The case of projective line bundles is considered by Loray and Pereira in \cite{MR2337401} in relation with transversely projective codimension one foliations. A natural question is to ask if this result could be recovered from a general version of Deligne's canonical extension that would respect $\mathfrak{g}$-connections, for $\mathfrak{g}$ a subalgebra of $\mathfrak{gl}_m(\C)$.

We should also mention the paper \cite{MR1223290} which describes in cohomological terms the obstructions to existence of linear and projective logarithmic connections on  logarithmic tangent bundles.

\textbf{Acknowledgements.} We are grateful to the referee for suggestions and helpful criticism. We acknowledge Carlos Simpson for informations about the state of the art for Theorem $\ref{liftingrepr}$.
We thank FIRB project ``Geometria differenziale e teoria geometrica delle funzioni", Marco Abate and Jasmin Raissy for their hospitality in Pisa. 
\section{Flat projective connections}
\subsection{Holomorphic connections} 
\begin{defi}
Let $m>0$.
Let $X$ be a complex manifold. A $\mathbb P^{m-1}$-bundle on $X$ is a holomorphically locally trivial bundle on $X$, $\pi :P\rightarrow X$ with fiber the complex $(m-1)$-dimensional projective space $\mathbb P^{m-1}$.
\end{defi}

\begin{defi}
Let $X$ be a complex manifold.
A  \textsl{holomorphic flat projective connection} $\nabla$ on the $\mathbb P^{m-1}$-bundle  $\pi : P\rightarrow X$ is a regular codimension $m-1$ holomorphic foliation on $P$, transversal to any fiber of $\pi$. 
Let $\star \in X$. The \textsl{monodromy representation} of $\nabla$ $$\rho : \pi_1(X,\star)\rightarrow Aut(\pi^{-1}(\star))$$ is defined as follows. For any loop $\alpha(t)$ in $X$ with base point $\star$, for any $y \in \pi^{-1}(\star)$, there is a unique lifting path $\tilde{\alpha}_y(t)$ of $\alpha(t)$, with $\tilde{\alpha}_y(0)=y$ and contained in a leaf of $\nabla$; we set $\rho(\alpha)$ to be the automorphism of $\pi^{-1}(\star)$ which satisfies $\rho(\alpha)(y)=\tilde{\alpha}_y(1)$ for every $y\in \pi^{-1}(\star)$.
\end{defi}
 Strictly speaking, this map is an antirepresentation, but we maintain the usual shortcut of ``monodromy representation".
Also, in effective computations, we are led to use an isomorphism $\pi^{-1}(\star)\stackrel{\phi}{\simeq} \mathbb{P}^{m-1}$ and replace $\rho$ by $\tilde{\rho}:\pi_1(X,\star)\rightarrow \mathrm{PGL}_m(\C)$ given by $\tilde{\rho}(\alpha)=\phi \circ \rho(\alpha) \circ \phi^{-1}$. We also make the abuse of language of naming $\tilde{\rho}$ the monodromy representation of $\nabla$. As $\phi$ is arbitrary, $\tilde{\rho}$ is well defined only up to conjugation by an element of $\mathrm{PGL}_m(\C)$.

It is well known that any $\rho :\pi_1(X,\star)\rightarrow \mathrm{PGL}_m(\C)$ can be realized as the monodromy representation of a unique (up to bundle isomorphism) flat projective connection, see \cite[Chapter V \S 4]{MR824240}.

For any flat holomorphic linear connection $D: \mathbf{V}\rightarrow \Omega_X^1\otimes\mathbf{V}$ on a vector bundle $V$ over $X$, the foliation induced by horizontal sections (i.e. sections $s$ such that $Ds=0$)  on the total space $V$ descends to a flat projective connection $\nabla=\mathbb{P}(D)$ on $\mathbb{P}(V)$. We call $\mathbb{P}(D)$ the \textsl{projectivization} of $D$.

Any flat connection is locally the trivial one on the trivial bundle. For this reason, any flat projective connection is locally the projectivization of a flat linear connection. 
Also we have a form of uniqueness.

\begin{lem}\label{uniqueness}
Let $D_i$, $i=1,2$ be two flat holomorphic connections on the same vector bundle $V$, with equal traces $$tr(D_1)=tr(D_2): det(\mathbf{V})\rightarrow \Omega_X^1 \otimes det(\mathbf{V}).$$ Then $\mathbb{P}(D_1)=\mathbb{P}(D_2)$ if and only if $D_1=D_2$.
\end{lem}
\begin{proof}
Using local trivializations, it suffices to check the result  for the trivial bundle $V=\mathcal{O}^m$.

 Let $\omega=(\omega_{i,j})$ be a size $m$ square matrix with coefficients in $\Omega_X^1(X)$ and define $D(y)=dy-\omega \cdot y$ for any vector valued holomorphic function $y=(y_1,\ldots ,y_m)^t \in \mathcal{O}^m$, we suppose $D$ is flat, that is $d \omega=\omega\wedge \omega$. We have a system of differential equations that define $\mathbb{P}(D)$ in the affine chart $y_m\neq 0$, setting $z_i=y_i/y_m, i=1,\ldots,m-1$ we find: 
 
$$dz_i=\omega_{i,m}+z_i(\omega_{i,i}-\omega_{m,m})+\sum_{k=1,k\neq i}^{m-1} \omega_{i,k} z_k -\sum_{k=1}^{m-1}\omega_{m,k} z_i z_k,~~i=1,\ldots,m-1. $$

   We see that the coefficients $\omega_{i,k}, k \neq i $ of $\omega$ are determined by this system; so do the differences $\Delta_i:=\omega_{i,i}-\omega_{m,m}$.
 The family $(\Delta_i)$ and $trace(\omega)$ determine $m\cdot\omega_{m,m}=trace(\omega)-\sum_{i=1}^{m-1} \Delta_i$, and subsequently every $\omega_{i,i}$.
\end{proof}

\begin{prop}\label{trivlift}
Let $X$ be a complex manifold and $\nabla$ be a holomorphic flat projective connection on $\mathbb{P}(\mathcal{O}_X^m)$, then $\nabla=\mathbb{P}(D)$ for a unique holomorphic flat linear trace free connection $D: \mathcal{O}^m_X\rightarrow (\Omega_X^1)^m$. 
\end{prop}

By trace free we mean $D(y)=dy-\omega \cdot y$ with $trace(\omega)=0$.

\begin{proof}
We can cover $X$ with open sets $U_i$ such that the connection is trivializable on $U_i$; taking $U_i$ small enough, this means there exist holomorphic maps  $G_i : U_i \rightarrow \mathrm{SL}_m(\C)$ such that, for $\phi_i=id\times \mathbb{P} G_i$, $\nabla_{\vert{U_i}}=\phi_i^*(\mathbb{P}(d))$, where $d$ is the trivial linear connection on $\mathcal{O}_{U_i}^m$. Also we can define flat linear connections by $D_i:=\psi_i^*d$, where $\psi_i=id\times G_i$. The connections $D_i$ are trace free because so is $d$ and the matrices $G_i$ take values in $\mathrm{SL}_m(\C)$. For clarity, let us draw a commutative diagram.
 $$\xymatrix{
(U_i\times\mathbb{P}^{m-1},\nabla)\ar[d]^{\phi_i}&(U_i\times\C^m, D_i) \ar[l]_-{\mathbb{P}} \ar[d]^{\psi_i}\\
(U_i\times\mathbb{P}^{m-1},\mathbb{P}(d) )&(U_i\times\C^m,d )\ar[l]_-{\mathbb{P}} \\
  }$$

 Let $U_{i,j}:=U_i \cap U_j$. If  $U_{i,j}\neq \varnothing$, the connections ${D_i}_{\vert U_{i,j}}$ and ${D_j}_{\vert U_{i,j}}$ are both trace free connections on the trivial bundle with projectivization $\nabla_{\vert U_{i,j}}$, thus they are equal by Lemma \ref{uniqueness}. This means the connection $D_i$ extends to a flat holomorphic connection $D$ on the trivial rank $m$ vector bundle over $X$ with $\mathbb{P}(D)=\nabla$. We have proved existence of the sought $D$, uniqueness follows from Lemma \ref{uniqueness}.
\end{proof}
\begin{rem}
It is tempting to try to generalize Proposition $\ref{trivlift}$ replacing $\mathbb{P}(\mathcal{O}_X^m)$ by any projectivization of a rank $m$ vector bundle. However, this would mean the map $H^2(X,\C^*)\rightarrow H^2(X,\mathcal{O}^*)$ would be injective on the image of the obstruction map $H^1(X,\mathrm{PGL}_m(\C)) \rightarrow H^2(X,\C^*)$. The work \cite{MR691957} allows to see this cannot be the case on any Abelian variety of dimension $g$; for any $m>1$ of the form $m=r^g$, $r\in \mathbb{N}^*$.
\end{rem}
\subsection{Logarithmic extensions}
\begin{defi}\label{logproj}
Let $X$ be a complex manifold and $H$ an analytic hypersurface. 
Let $P\rightarrow X$ be a $\mathbb{P}^{m-1}$-bundle on $X$.
A \textsl{logarithmic flat projective  connection} on $P$, with poles in $H$, is a singular holomorphic codimension $m-1$ foliation $\nabla$ on $P$ with the following properties.
\begin{enumerate}
\item
 The foliation $\nabla$ restricts to a holomorphic flat projective connection on $P_{\vert X \setminus H}$.
 
 \item For every $x \in H$, there exists a neighborhood $U$ of $x$ and a flat logarithmic connection $D$ on the trivial rank $m$ vector bundle over $U$ with poles in $H$, such that  there exists a bundle isomorphism $\phi : P_{\vert U} \rightarrow \mathbb{P}(\mathcal{O}^{m}_U)$ satisfying $\phi^*\mathbb{P}(D)_{\vert U\setminus H}=\nabla_{\vert P_{\vert U\setminus H}}$.
 \end{enumerate}
 
 We define the \textsl{monodromy representation} of $\nabla$ to be the one of $\nabla_{\vert P_{\vert X \setminus H}}$.
\end{defi}
Let us introduce a property $\mathcal{P}_m(M)$ for an element $M \in \mathrm{PGL}_m(\C)$.
$$\mathcal{P}_m(M): \left \lbrace \begin{array}{c} \mbox{ For any }\tilde{M}\in \mathrm{GL}_m(\C)\mbox{  with }\mathbb{P}(\tilde{M})=M,\\
\mbox{for any two eigenvalues }\lambda_1,\lambda_2\mbox{ of }\tilde{M},\\\lambda_1^m=\lambda_2^m \Rightarrow \lambda_1=\lambda_2.\end{array} \right .$$
Of course it suffices to check this condition for only one lift $\tilde{M}\in \mathrm{GL}_m(\C)$.

For $X$ a complex manifold and $H$ a hypersurface in $X$, if $H_j$ is a component of $X$ we call $\alpha \in \pi_1(X\setminus H,\star)$
a simple loop around $H_j$ if $\alpha$ is conjugate by a path to $(x,z)(t)=(e^{2i\pi t},z_0)$ for a coordinate patch $(x,z_1,\ldots,z_l)$ of $X$  centered at a point of $\{x=0\} \subset H_i$.
Our realization result is the following.

\begin{thm}\label{realization}
Let $X$ be a complex manifold. Let $H$ be a normal crossing analytic hypersurface on $X$.
Let $\rho : \pi_1(X\setminus H,\star)\rightarrow \mathrm{PGL}_m(\C)$ be an antirepresentation.
Suppose, for every simple loop $\alpha \in \pi_1(X\setminus H,\star)$ around any component  of $H$,  we have $\mathcal{P}_m\left(\rho(\alpha)\right)$.

Then $\rho$ is the monodromy representation of a flat projective logarithmic connection with poles in $H$.
\end{thm}

Before proving Theorem \ref{realization}, we introduce local models and study their symmetries.

For $A_1\in M_m(\C)$, and coordinates $(x_1,\ldots,x_n)\in \C^n$ set for $y$ a local section of $\mathcal{O}^m$,
$$D_{A_1}(y):=dy-\frac{A_1 dx_1}{x_1}y.$$ The monodromy of $D_{A_1}$ is generated by $exp(2 i \pi A_1)$.

More generally, for a family $A=(A_1,\ldots,A_k)\in M_m(\C)^k$ of commuting matrices with $k\leq n$, we can define a flat connection $D_A$  on $\mathcal{O}^m$ by
$$D_{A}(y):=dy-\sum_{i=1}^k \frac{A_i dx_i}{x_i}y.$$
We say that $A_1\in M_m(\C)$ is \textsl{nonresonant} if for any pair $\mu_1,\mu_2$ of eigenvalues of $A_1$, $\mu_1-\mu_2 \not \in \mathbb{N}^*$. 

\begin{lem} \label{extension}Let $A_1\in M_m(\C)$ be nonresonant.
Let $\tau_1,\tau_2$ be two size $m$ square matrices of holomorphic  $1$-forms defined on a neighborhood $U$ of $0$ in $\C^n$.  Let $\omega_i:=A_1\frac{dx_1}{x_1}+\tau_i$, $i=1,2$.

For $i=1,2$; let $D_i(y):=dy-\omega_i\cdot y$ and suppose $D_i$ is a flat connection (\textit{i.e.} $d \omega_i=\omega_i \wedge \omega_i$).

 Then any isomorphism between the connections ${D_i}_{\vert U \setminus \{x_1=0\}}$ extends to an isomorphism on the whole of $U$.
\end{lem}

\begin{proof} This is an easy modification of the proof of \cite[Lemme $3$]{MR2077648}; see also \cite[Prop. $5.2.d)$]{MR0417174}.
\end{proof}

We will make use of the following normalisation result.
\begin{thm}[Poincar\'e]\label{Poincare}
Let $\omega=(\omega_{i,j})_{1\leq i,j\leq m}$, be a matrix of meromorphic $1$-forms on a neighborhood $U$ of $0$ in $\C^n$, with coordinates $x_1,\ldots,x_n$. Suppose $d \omega=\omega\wedge \omega$.
Suppose the only pole of $\omega$ is $x_1=0$ and $\omega=A_1 \frac{dx_1}{x_1}+\tau$ for a holomorphic matrix $1$-form $\tau$ and $A_1\in M_m(\C)$ a non resonant matrix, then there exists a neighborhood $V\subset U$ of $0$, such that the connection $D$
on $\mathcal{O}^m_V$ defined by $D(y)= dy-\omega\cdot y$ is isomorphic to ${D_{A_1}}_{\vert V}$.
\end{thm}

\begin{proof} The result for only one variable is well known and allows to suppose $\tau_{\vert (x_2,\ldots, x_n)=0}=0$. Then, coincidence of monodromy yields the required isomorphism outside $x_1=0$. Finally, our nonresonance assumption allows to extend the isomorphism holomorphically at $x_1=0$ by Lemma $\ref{extension}$.  
\end{proof}
\begin{lem}\label{extensionproj}
Let $A_1\in \mathrm{M}_m(\C)$. Suppose $m A_1$ is nonresonant.
Let $U$ be a neighborhood of $0$ in $\C^n$ and suppose we have an automorphism $\phi$
of the holomorphic projective connection $\nabla=\mathbb{P}({D_{A_1}}_{\vert U \setminus\{x_1=0\}})$, then $\phi$ extends to an automorphism  of the trivial $\mathbb{P}^{m-1}$-bundle over $U$.
\end{lem}

\begin{proof}
We can suppose $U$ is a polydisk.
The automorphism $\phi$ is of the form $(x,z) \mapsto (x,G(x)\cdot z)$ for a holomorphic function $$G : U \setminus \{x_1=0\} \rightarrow \mathrm{PGL}_m(\C).$$ We can lift this map to a multivalued holomorphic function from $U \setminus \{x_1=0\}$ to $\mathrm{SL}_m(\C)$. Its monodromy is generated by $M\mapsto \lambda M$, $\lambda$ satisfying $\lambda^m=1$.
Thus, if the covering $\pi : V \rightarrow U$ is defined by $ (u_1,\ldots,u_n)\mapsto (x_1,\ldots,x_n)=(u_1^m,u_2,\ldots,u_n)$, there exists a holomorphic function $\tilde{G} : V\setminus \{u_1=0\} \rightarrow \mathrm{SL}_m(\C)$ satisfying $\mathbb{P}(\tilde{G}(u))=G\circ\pi(u)$. This function induces an automorphism $(u,z)\mapsto (u,\tilde{G}(u)\cdot z)$ of the 
pull-back $\mathbb{P}({D_{mA_1}}_{\vert V \setminus\{u_1=0\}})$ of $\nabla$ by $\pi$. Also, by Lemma \ref{uniqueness}, $(u,y)\mapsto(u,\tilde{G}(u)\cdot y)$ is an automorphism of ${D_{mA_1}}_{\vert V \setminus\{u_1=0\}}$.
By hypothesis $mA_1$ is nonresonant, thus $\tilde{G}$ and $\tilde{G}^{-1}$ extend to holomorphic functions on $V$, by Lemma $\ref{extension}$. Thus $G$ also extends as desired.
\end{proof}

\noindent \textit{Proof of Theorem $\ref{realization}$.}
Let $U_0=X\setminus H$. Let $(H_i)_{i \in I}$ be the components of $H$. Let $\alpha_i \in \pi_1(U_0,\star)$ be a simple loop turning counterclockwise around $H_i$. For any $i$, choose a lift $M_i\in \mathrm{SL}_m(\C)$  for $\rho(\alpha_i)$ and let $A_i\in \mathrm{M}_m(\C)$, with real parts of its eigenvalues $\mu$ satisfying $0\leq \Re(\mu)<1$ be  such that $exp(2 i \pi A_i)=M_i$. Thanks to $\mathcal{P}(\rho(\alpha_i))$,  $m A_i$ is automatically nonresonant.

Let $p\in H$, and let $H_{i_j},j=1,\ldots, k$ be the components of $H$ which contain $p$. Because of normal crossings, there exists a neighborhood $U_p$ of $p$ and a chart $f_p: U_p\stackrel{\sim}{\rightarrow}  \Delta$ to $\Delta=\{(x_1,\ldots,x_n)\in\C^n,|x_i|<2\}$ such that $H_{i_j}\cap U_p$ is sent to $x_j=0$ by $f_p$. Let $\star_p=f_p^{-1}(1,\ldots,1)$.
The fundamental group $\pi_1(U_p,\star_p)$ is abelian, generated by the loops $(\beta_j)_{j=1,\ldots,k}$ defined by $x_j(\beta_j(t))=\exp(2i\pi t)$, $x_l(\beta_j(t))=1$ for $l\neq j$; $t\in[0,1].$
Choose a path $\tau$ in $U_0$ from $\star$ to $\star_p$, the loop $\tau \beta_j \tau^{-1}$ defines an element $\gamma_j \in \pi_1(U_0,\star)$ conjugate to $\alpha_{i_j}$. The elements $(\gamma_j)_{j=1,\ldots, k}$ commute pairwise.

 Choose lifts $N_j\in \mathrm{SL}_m(\C)$ of $\rho(\gamma_j)$ of the form $N_j=G_jM_{i_j}G_j^{-1}$, $G_j\in \mathrm{GL}_m(\C)$. The abelianity of $<(\mathbb{P}N_j)_j> \subset\mathrm{PGL}_m(\C)$ gives $N_{j_1} N_{j_2} N_{j_1}^{-1}=\lambda N_{j_2}$ with $\lambda^m=1$, but $\mathcal{P}_m(\rho(\gamma_{j_2}))$ yields $\lambda=1$  and we have abelianity of $<~(N_j)_j>\subset \mathrm{GL}_m(\C)$.

 Because of the latter abelianity and \cite[Lemme $2$]{MR2077648}, 
 there exists a linear flat connection $D_p$ on $\mathcal{O}^m_{U_p}$ with residues $B_{i_j}=G_jA_{i_j}G_j^{-1}$ on $H_{i_j}$ and monodromy $\hat{\rho}_p: \pi_1(U_p,\star_p)\rightarrow \mathrm{SL}_m(\C)$ given by $\hat{\rho}_p(\beta_j)=N_j$; set $\nabla_p=\mathbb{P}(D_p)$. On $U_0$ take $\nabla_0$ a projective flat connection with monodromy $\rho$. We denote $P_p$, $P_0$ the underlying bundles of $\nabla_p$,$\nabla_0$ respectively. Also, set $U_{p,0}:= U_p\cap U_0=U_p \setminus H$, $U_{p,q}:= U_p\cap U_q$.
 
 Consider $\rho_p^0,\rho_p:\pi_1(U_p,\star_p)\rightarrow \mathrm{PGL}_m(\C)$, the respective monodromies of ${\nabla_0}_{\vert U_{p,0}}$ and $\nabla_p$. We have $\rho_p^0(\beta_j)=M_{\tau}\mathbb{P}{(N_j)}M_{\tau}^{-1}=M_{\tau}\rho_p(\beta_j)M_{\tau}^{-1}$, where $M_{\tau}$ is the holonomy of $\nabla_0$ over the path $\tau$. 
Because of this conjugation, we have an isomorphism $\phi_{p,0} : {P_0}_{\vert U_{p,0}}\rightarrow {P_p}_{\vert U_{p,0}}$, such that $\phi_{p,0}^*{\nabla_p}_{\vert U_{p,0}}={\nabla_0}_{\vert U_{p,0}}$.

 Define $\phi_{0,p}=\phi_{p,0}^{-1}$ and denote $H_0$ the set of singular points of $H$.

By Theorem $\ref{Poincare}$ and Lemma $\ref{extensionproj}$, the composition $\phi_{p,0}\circ \phi_{0,q}$ extends to an isomorphism ${P_q}_{\vert U_{p,q}\setminus H_0}\simeq{P_p}_{\vert U_{p,q}\setminus H_0}$, then it extends to $\phi_{p,q} :{P_q}_{\vert U_{p,q}}\simeq {P_p}_{\vert U_{p,q}}$ because $H_0$ has codimension $>1$ in  $U_{p,q}$. By definition, the functions $\phi_{i,j}$ satisfy the cocycle relations and define a $\mathbb{P}^{m-1}$-bundle $P$ over $X$. The local connections $\nabla_{j}$ on $P_j$ satisfy $\phi_{i,j}^*{\nabla_i}_{\vert U_{i,j}}={\nabla_j}_{\vert U_{i,j}}$ and give the sought flat projective logarithmic connection on $P$.
\hfill$\square$\vspace{0.06cm}
\begin{rem}
In this proof, we have used conditions $\mathcal{P}_m(\rho(\alpha_{i}))$  for two reasons: to lift the local monodromy of $\nabla_0$ at any point of the polar locus and to
 extend bundle isomorphisms to the polar locus.
 As can be seen from Definition $\ref{logproj}$, the local liftings must exist to extend $\nabla_0$ to $H$. The proof of \cite[Lemma 2.3]{MR691957} shows an obstruction to local lifting. The extension condition is more subtle. We have chosen the property $\mathcal{P}_m$ to have a simple statement; it seems of interest to see how these conditions can be weakened. 
\end{rem}

\section{Lifting result}
\vspace{-0.2cm}
We will (re)prove the following.
\begin{thm}\label{liftingrepr}
Let $X$ be an irreducible projective complex variety and $H$ an algebraic hypersurface in $X$. Let $\star$ be a smooth point of $X\setminus H$. 
For any representation $\rho : \pi_1(X\setminus H,\star)\rightarrow \mathrm{PSL}_m(\C)$, there exists a hypersurface $H_{\rho}$ with $H \subset H_{\rho}, \star \not \in H_{\rho}$ and a generically finite morphism $f_{\rho} : (Y_{\rho},\star_{\rho}) \rightarrow (X,\star)$ of projective varieties with basepoints, \'etale in the neighborhood of $\star_{\rho}$,  such that $Y_{\rho}$ is smooth and  the pull-back $$f_{\rho}^*\rho : \pi_1(Y_{\rho}\setminus f_{\rho}^{-1}(H_{\rho}),\star_{\rho})\rightarrow \mathrm{PSL}_m(\C)$$ lifts to $\mathrm{SL}_m(\C)$, that is $f_{\rho}^*\rho=\mathbb{P}\hat{\rho}$, for a representation $$\hat{\rho} : \pi_1(Y_{\rho}\setminus f_{\rho}^{-1}(H_{\rho}),\star_{\rho})\rightarrow \mathrm{SL}_m(\C).$$
\end{thm}

\begin{proof}
By resolution of singularities, after some birational morphism we can suppose $X$ is smooth and $H$ is normal crossing; we make this assumption in the sequel.

 Let $(H_i)$ be the irreducible components of $H$ and let $\alpha_i\in \pi_1(X\setminus H,\star)$ be a simple loop around $H_i$. Take a lift $M_i\in \mathrm{SL}_m(\C)$ for $\rho(\alpha_i)$. Consider the finite set $S_i$ whose elements are finite order quotients $\lambda/\mu$ of eigenvalues $\lambda,\mu$ of $M_i$. Set $S:=\cup_i S_i$ and let $O\subset \N^*$ be the set given by the orders of the elements of $S$. Let $\nu:=\mathrm{lcm}(O)$. Let 
 $r : (X_1,\star_1) \rightarrow (X,\star)$ be a finite morphism, \'etale in the neighborhood of $\star_1$, with ramification indices over $H_i$ equal to multiples of $\nu$ and such that $H_1=r^{-1}(H)$ is a normal crossing hypersurface; existence of such an $r$ is given, for example, by \cite[Theorem 17]{MR622451}.

 Then, we can apply our realization Theorem $\ref{realization}$: there exists a flat projective logarithmic connection $\nabla$ with poles in $H_1$ with monodromy $r^*\rho$. Let $P$ be the underlying analytic locally trivial $\mathbb{P}^{m-1}$-bundle of $\nabla$ and take $\star_1\in r^{-1}(\star)$. By Serre, \cite[Th\'eor\`eme $3$ p. $34$]{SerreChev58}, $P$ is the analytification of an algebraic locally isotrivial $\mathbb{P}^{m-1}$-bundle:  $\star_1$ has a Zariski neighborhood $U_1\subset X_1$, $U_1=X_1\setminus \tilde{H}_1$ such that there exists a finite \'etale algebraic covering $q : (U_2,\star_2)\rightarrow (U_1,\star_1)$ satisfying that $q^*P$ is trivial. We can suppose $H_1\subset \tilde{H}_1$ and $U_1$ is affine, which we do.
 
  Then, we have an embedding $U_2\subset Y$, $\overline{U_2}=Y$, in a smooth projective $Y$ such that the algebraic map $q$ extends to a morphism $q :Y\rightarrow X_1$. 
  By triviality of $q^*(P)_{\vert U_2}$ and Proposition $\ref{trivlift}$, $\nabla_2:=q^*(\nabla_{\vert U_1})$ lifts to a trace free flat linear connection over $U_2$. For this reason, the monodromy representation of $\nabla_2$ lifts to $\mathrm{SL}_m(\C)$. By construction this monodromy is $q^* r^* \rho$. Hence, if we define, $\star_{\rho}:=\star_2$, $Y_{\rho}:=Y$, $f_{\rho}:=r\circ q$ and $H_{\rho}$ to be the codimension $1$ part of $r\circ q(Y\setminus U_2)$, we have the situation announced in the theorem. 
\end{proof}
\bibliographystyle{smfalpha}
\bibliography{biblio}
\end{document}